\documentclass[11pt]{amsart}
\usepackage{amsmath}
\usepackage{amssymb}
\usepackage{euscript}
\usepackage{enumerate}
\usepackage{color}
\usepackage{graphics}

\newtheorem{theorem}[subsubsection]{Theorem}
\newtheorem{lemma}[subsubsection]{Lemma}
\newtheorem{cor}[subsubsection]{Corollary}

\newtheorem{prop}[subsubsection]{Proposition}

\newtheorem{definition}[subsubsection]{Definition}

\numberwithin{equation}{section}

\title[Regularity conditions of 3D Navier-Stokes]{Regularity conditions of 3D Navier-Stokes flow in terms of large spectral components}
\begin{document}

\author{Namkwon Kim}
\address{Department of Mathematics, \newline \indent Chosun University, \newline \indent Kwangju 501-759, Korea}
\email{kimnamkw@chosun.ac.kr}
\author{Minkyu Kwak}
\address{Department of Mathematics, \newline \indent Chonnam National University, \newline \indent Kwangju 501-757, Korea}
\email{mkkwak@chonnam.ac.kr}
\author{Minha Yoo}
\address{National Institute for Mathematical Science, \newline \indent Daejeon 305-811, Korea}
\email{mhyoo0702@gmail.com}
\maketitle
\begin{abstract}
We develop Ladyzhenskaya-Prodi-Serrin type spectral regularity criteria for 3D incompressible Navier-Stokes equations in a torus.
Concretely, for any $N>0$, let $w_N$ be the sum of all spectral components of the velocity fields whose all three wave numbers are greater than $N$ absolutely.  Then, we show that for any $N>0$, the finiteness of the Serrin type norm of $w_N$ implies the regularity of the flow.
It implies that if the flow breaks down in a finite time, the energy of the velocity fields cascades down to the arbitrarily large spectral components of $w_N$ and corresponding energy spectrum, in some sense, roughly decays slower than $\kappa^{-2}$ .
\end{abstract}

\section{Introduction}\label{intro}
In this paper, we are concerned with the 3D incompressible Navier-Stokes equations in a flat torus:
\begin{equation} \label{eq-main-ns}
\begin{cases}
\displaystyle\frac{\partial \bold u}{\partial t}(x,t) - \nu \Delta \bold u(x,t) + \bold u(x,t) \cdot \nabla \bold u(x,t)  + \nabla p(x,t) = \bold f(x,t),\\
\text{div } \bold u(x,t) = 0,\\
\bold u(x,0) = \bold u_0(x)
\end{cases}
\end{equation}
for $x \in \mathbb T^3 = \mathbb R^n / \mathbb Z^n$.  Here, ${\bold u}(\cdot, t) :  {\mathbb T^3} \to {\mathbb R^3}$ is the velocity fields
and $p(\cdot,t) :  {\mathbb T^3} \to {\mathbb R}$ is the pressure.
 If $\bf{u}_0$ is smooth enough, there exists a unique solution of \eqref{eq-main-ns} for $0<t < T$ for some $T>0$(See \cite{CF} and references therein).   Then, there arises a natural question whether $T=\infty$ for every smooth $u_0$ or not.   This question has not been
settled yet.   However,  if the solution satisfies one of the following conditions up to $t=T$ for some $T>0$,
\begin{enumerate}
\item \label{eq-serrin-1}
$\nabla \bold u \in L^r(0,T;L^q(\mathbb T^3)),\, \displaystyle\frac{2}{r} + \frac{3}{q} = 2,\, q \in [3/2,\infty]$,
\item \label{eq-serrin-3}
$\bold u \in L^r(0,T;L^q(\mathbb T^3)),\, \displaystyle\frac{2}{r} + \frac{3}{q} = 1,\, q \in [3,\infty]$.
\end{enumerate}
then the solution becomes smooth for $0<t \leq T$\cite{L,P,SS,S}.  The above conditions are called the Serrin conditions.
The same is true for $\mathbb R^3$  instead of the flat torus $\mathbb T^3$.

We are interested in improving the above Serrin conditions in terms of spectral decomposition.
Some parts of the Fourier expansion of the velocity fields can be essentially considered as a two dimensional flow and we expect that a certain genuinely three dimensional part of the velocity fields plays the crucial role for the regularity of the flow.  To be specific, let
\begin{equation} \label{def-ck}
\bold u = \sum_{k \in \mathbb Z^3}\bold c^k e^{2\pi ik\cdot x}
\end{equation}
be the Fourier expansion of $\bold u$.   For any $N>0$, we call
\begin{equation}  \label{def-wN}
\bold w_N \equiv Q_N (\bold u) \equiv \sum_{|k_1|, |k_2|,|k_3| > N,} \bold c^k e^{2\pi ik\cdot x}
\end{equation}
a genuine 3D part of $\bold u$ cut at the mode $N$.
Then, we shall show that if, for any $q >1$,
\begin{equation}  \label{eq-serrinN-1}
 \liminf_{N\to \infty} \| Q_N (\bold u ) \|_{L^{q,r}(0,T;\mathbb T^3)} < \infty, \quad 3/q+2/r=1
\end{equation}
\mbox{ or }
\begin{equation} \label{eq-serrinN-3}
\liminf_{N\to \infty} \| \nabla Q_N (\bold u )\|_{L^{q,r}(0,T;\mathbb T^3)} < \infty,  \quad 3/q+2/r=2,
\end{equation}
then the regularity of $\bold u$ up to $t=T$ is guaranteed.

If $\bf{u}_0$ is smooth and $\bold u$ breaks down firstly at a finite time, $t=T$,  the above condition actually means that $L^{q,r}$ norm
of $\bold w_N$ blows up as $t\to T$ for any $N>0$.  Taking $q=3$, $r=\infty$, this, in particular,
implies the energy density of $\bold w_N$,
\[ E(\kappa) \equiv \sum_{|k| =\kappa} |\bold c^k|^2 k^2,  \]
is not uniformly bounded by $ \kappa^{-\delta}$ for any $\delta > 2$ near $t=T$.
In many statistical turbulence theories,  the energy spectrum is expected to behave like $\kappa^{-a}$ with $a \in [1,2]$.   For example, in Markovian Random Coupling Model and Quasi-Normal Markovian Approximation theory, $E(\kappa)\sim \kappa^{-2}$ is expected(see chapter $VII$ of \cite{ML}) and, in Kolomogorov's homogeneous turbulent theory, $E(\kappa) \sim \kappa^{-5/3}$ is expected\cite{Kol,ML}.
In this sense, our result actually supports that the breakdown of the solutions is related with the turbulent aspect of the genuine 3D part of the velocity fields.

\section{Preliminary}\label{mainthm}
We define $\mathcal V$ the set of all smooth functions $\bold \varphi$ which are divergence free and mean zero, i.e., $\mathcal V = \{ \bold \varphi \in C^\infty(\mathbb T^3)^3 : \int \varphi dx = 0, \text{ and }\text{div} \bold \varphi = 0 \}$. We also define $H$ and $V$ by the closure of the set $\mathcal V$ with respect to norms $\| \cdot \|_{L^2(\mathbb T^3)}$ and $\| \cdot \|_{H^1(\mathbb T^3)}$ respectively.
In the Fourier expansion, the function spaces, $H$ and $V$ can be represented as follows:
\begin{equation}
\begin{aligned}
H &= \{\bold h \in L^2(\mathbb T^3) :  \hat{\bold h}^0 = 0, k \cdot \hat{\bold h}^k = 0 \text{ for all } k \in \mathbb Z^3\},\\
V &= \{\bold h \in H^1(\mathbb T^3) : \hat{\bold h}^0 = 0, k \cdot \hat{ \bold h}^k = 0 \text{ for all } k \in \mathbb Z^3\}.
\end{aligned}
\end{equation}


%

We let $P$ be the orthogonal projection of $L^2(\mathbb T^3)$ to $H$. Then the weak solution of Navier-Stoke equation is given in the following way,
\begin{definition}
A weak solution of the Navier-Stokes equation \eqref{eq-main-ns} is a function $\bold u \in L^2(0,T;V) \cap L^\infty(0,T;H)$ satisfying $\displaystyle\frac{du}{dt} \in L^1_{\text{loc}}([0,T);V^\prime)$ and
\begin{equation} \label{eq-ns-weak} \begin{cases}
\left(\displaystyle\frac{d\bold u}{dt}, \bold \varphi \right) + \left( A \bold u, \bold \varphi \right) + b(\bold u,\bold u,\bold \varphi) = \left( \bold f, \bold \varphi \right),\qquad \forall\bold \varphi \in V,\\
\bold u(0) = \bold u_0.
\end{cases} \end{equation}
\end{definition}
where $(\cdot,\cdot)$ is the standard inner product of $L^2(\mathbb T^3)$, $A$ is the Stokes operator given by $A = P(-\Delta) : \mathcal D(A) \rightarrow H$, and $b(\cdot,\cdot,\cdot)$ is a trilinear form given by $b(\varphi, \psi, \eta) = \sum_{i,j,k=1}^3\int_{\mathbb R^n} \varphi_i \, \partial_i \psi_j \, \eta_j dx$.  We remark that the test function $\varphi$ in fact can be chosen in $L^\infty(0,T;V)$
since all terms in \eqref{eq-ns-weak} remain meaningful.
It is well-known that there exists a weak solution of Navier-Stokes equations\cite{Le}:
\begin{theorem}
There exists at least a weak solution $\bold u$ of the Navier-Stokes equation \eqref{eq-main-ns} for every $\bold u_0 \in H$, $\bold f \in L^2(0,T;V^\prime)$.
\end{theorem}
This Leray weak solution becomes smooth and unique up to $t=T$ if $\bold u$ belongs to the Serrin class, \eqref{eq-serrin-1} or \eqref{eq-serrin-3} as we mentioned earlier .

Let $c^k(t)$ be the Fourier expansion of $\bold u$ as in \eqref{def-ck} and $\lambda_k = 4\pi^2 |k|^2$. Then $A \bold u$ is
\begin{equation}
A \bold u = \sum_{k \in \mathbb Z^n} \bold  c^k \lambda_k e^{2\pi i k \cdot x}.
\end{equation}
Also, for any $0\leq s<1$, we define $A^s$ by
\begin{equation}
A^s \bold u = \sum_{k \in \mathbb Z^n} \bold  c^k \lambda_k^s e^{2\pi i k \cdot x}.
\end{equation}
For any $N>0$, and the solution $\bold u$ of \eqref{eq-main-ns}, we define $\bf{w}_N$
as in \eqref{def-wN} and
\begin{equation} \label{eq-main-vw}
\begin{aligned}
\bold v_N &= \bold u- \bf{w}_N = \bold v^1 + \bold v^2 + \bold v^3\\
        &= \sum_{ |k_1| \le N} c^k e^{2\pi ik\cdot x} + \sum_{ |k_1| > N, |k_2 | \le N } c^k e^{2\pi ik\cdot x} \\
        &\quad+ \sum_{ |k_1|, |k_2| >N, |k_3| \le N} c^k e^{2\pi ik\cdot x}.
\end{aligned}
\end{equation}
For $\bf{w}_N$, we have the following usual Sobolev embedding theorem\cite{BO}.
\begin{theorem} \label{thm-tec-sob1}
Let $n>1$ be an integer, $2<q<\infty$, and $s>0$ satisfy
\begin{equation}
\frac{2s}{n} = \frac{1}{2} - \frac{1}{q}.
\end{equation}
Then for any $h \in L^2(\mathbb T^n)$, we have the following,
\begin{equation}
\| h\|_{L^q} \le C \left\{ \| (-\Delta)^s h \|_{L^2} + \| h \|_{L^2} \right\}
\end{equation}
where $C$ is a constant depending only on $n$, $s$, and $q$ and, for $0<s<1$,
\begin{equation}
(-\Delta)^{s} h(x) =  \sum_{k \in \mathbb Z^n} \widehat h^k \lambda_k^{s} e^{2\pi i k \cdot x}
\end{equation}
and $\lambda_k = 4 \pi^2 |k|^2$ as before.
\end{theorem}
$\bf{v}_N$ possesses infinite number of spectral components and is of 3-dimensional.
However, they satisfies  the following 2-dimensional Sobolev embeddings.
\begin{lemma} \label{lem-tec-v1}
Given any $s$, $0 \le 2s <1$,
\begin{equation} \label{eq-tec-v1}
\| \bold v^i \|_{L^q} \le C N^{1/2} \| A^s  \bold v^i \|_{L^2}, \quad i=1,\cdots ,3
\end{equation}
for $\frac{1}{q} = \frac{1}{2} - s$. Also,
\begin{equation} \label{eq-tec-v4}
\| \bold v_N \|_{L^q} \le C N^{1/2} \| \bold v_N \|_{L^2}^{2/q} \| A^{1/2} \bold v_N \|_{L^2}^{1 -2/q}.
\end{equation}
\end{lemma}
\begin{proof}
We first show \eqref{eq-tec-v1}.  Since the proof is exactly parallel, we only show it for $\bold v^1$.
Let $x=(x_1, \overline x) \in \mathbb R\times \mathbb R^2$, $k=(k_1,\overline k) \in \mathbb Z \times \mathbb Z^2$, and
\[ h_{k_1}(\overline x) = \sum_{\overline k \in \mathbb Z^2} \bold c^{(k_1,\overline k)}(t) e^{2 \pi i \overline k \cdot \overline x} \]
where $\bold c^{(k_1,\overline k)}$ is as in \eqref{def-ck}.

Then $\bold v^1$ is represented as follows,
\begin{equation}
\bold v^1(x) = \sum_{|k_1| \le N} e^{2\pi i k_1 x_1} \, \sum_{\overline k \in \mathbb Z^2} \bold c^{(k_1,\overline k)} e^{2 \pi i \overline k \cdot \overline x} = \sum_{|k_1| \le N} e^{2\pi i k_1 x_1} \, h_{k_1}(\overline x).
\end{equation}

From the definition of $h_{k_1}$, the fractional derivative of $h_{k_1}$ is given as
 \begin{equation}
(-\Delta)^s h_{k_1} (\overline x) = \sum_{\overline k \in \mathbb Z^2} \bold c^{(k_1,\overline k)}(t) (\lambda_{\overline k})^s e^{2 \pi i \overline k \cdot \overline x}
\end{equation}
where $\lambda_{\overline k} = 4\pi^2 |\overline k|^2$.

Now we are comparing the $L^2$-norm of $\| A^s  \bold v^1(x) \|_{L^2}^2$ and $\sum_{|k_1| \le N} \| (-\Delta)^s h_{k_1} \|_{L^2}$. Since
\begin{equation} \begin{aligned}
&\| A^s  \bold v^1(x) \|_{L^2}^2 = \sum_{|k_1| \le N} \, \sum_{\overline k} \bold (|k_1|^2 + |\overline k |^2)^s | c^{(k_1,\overline k)} |^2\\
&\qquad= \sum_{|k_1| \le N} \,\left\{ \sum_{\overline k \neq 0} \frac{(|k_1|^2 + |\overline k |^2)^s}{|\overline k|^{2s}} |\overline k|^{2s} | c^{(k_1,\overline k)} |^2 + |k_1|^{2s} | c^{(k_1,0)} |^2 \right\},\\
\end{aligned} \end{equation}
and $|k_1| \geq 1$ in the latter term due to the average zero property of $\bold u$,
we have
\begin{equation}
\sum_{|k_1| \le N} \left\{ \| (-\Delta)^s h_{k_1} \|_{L^2}^2 + \| h_{k_1} \|_{L^2}^2 \right\} \le C \| A^s  \bold v^1(x) \|_{L^2}^2.
\end{equation}
In particular, by the Minkowski inequality, the Cauchy-Schwarz inequality, and Theorem \ref{thm-tec-sob1} with $n=2$, we have
\begin{equation}
\begin{aligned}
\| \bold v^1 \|_{L^q}^2 &\le \left(\sum_{|k_1| \le N} \| h_{k_1} \|_{L^q} \right)^2 \le (2N+1) \sum_{|k_1| \le N} \| h_{k_1} \|_{L^q}^2 \\
& \le C N\sum_{|k_1| \le N} \left\{ \| (-\Delta)^s h_{k_1} \|_{L^2}^2 + \| h_{k_1} \|_{L^2}^2 \right\} \\
&\le C N \| A^s  \bold v^1(x) \|_{L^2}^2
\end{aligned}
\end{equation}
for $s=1/2-1/q$.  Here, $C$ is a constant depending only on $q$.

Next, we show \eqref{eq-tec-v4}.
By the $l^{1/s}-l^{1/(1-s)}$ duality, we have
\begin{align*}
\| A^{s} \bold v_N \|_{L^2}^2 &= \sum_k |\bold c^k|^2 |\lambda_k|^{2s}
                   =  \sum_k (|\bold c^k|^{2} |\lambda_k|)^{2s} \; |\bold c^k|^{2(1-2s)} \\
                   & \leq (\sum_k |\bold c^k|^{2} |\lambda_k|)^{2s} (\sum_k |\bold c^k|^{2} )^{1-2s} \\
              & = \| A^{1/2} \bold v_N \|_{L^2}^{4s} \| \bold v_N \|_{L^2}^{2-4s}.
\end{align*}
Hence, together with \eqref{eq-tec-v1}, we finish the proof.
\end{proof}
We finish this section by introducing the theorem.
\begin{theorem} [\cite{G}]\label{thm-tec-sob2}
Let $n$ be any positive integer and $q$, $r$ be constants satisfying $q \in [r, nr/(n-r)]$ if $r \in [1,n)$, and
$q \in [r,\infty)$, if $r \ge n$.
Then, for all $\varphi$ such that $\varphi \in L^r(\mathbb T^n)$ and $\nabla \varphi \in L^r(\mathbb T^n)$, we have
\begin{equation} \label{eq-thm-tec-sob2}
\| \varphi \|_{L^q} \le C(n,r,q) \| \varphi \|_{L^r}^{a} \| \nabla \varphi \|_{L^r}^{1-a}
\end{equation}
where
$\displaystyle\frac{1}{q} = \frac{a}{r} + (1-a)\left( \frac{1}{r} - \frac{1}{n} \right)$.
\end{theorem}

The inequality \eqref{eq-tec-v4} is in fact the inequality \eqref{eq-thm-tec-sob2} with $n=2$.
This is why $\bold v$ has 2-dimensional properties.

\section{Spectral Serrin condition}\label{main}

We now state our main theorem and show it here.
\begin{theorem} \label{thm-main}
Suppose that $\bold u$ is a weak solution of the 3D Navier-Stokes equation \eqref{eq-main-ns} and satisfies one of the Serrin conditions \eqref{eq-serrinN-1} - \eqref{eq-serrinN-3}. Then $\bold u$ is a classical solution up to $t=T$.
\end{theorem}

\subsection{Proof for \eqref{eq-serrinN-3} with $q>3/2$}

By \eqref{eq-serrinN-3}, there exists $N>0$ such that $\nabla \bold w_N \in L^r(0,T;L^q(\mathbb T^3)),\, \displaystyle\frac{2}{r} + \frac{3}{q} = 2,\, q \in (3/2,\infty]$.
For a weak solution, for almost every $0<s<T$, $\bold u (s) \in V$.  Then, for some $S \in (s, T]$, $\bold u$ becomes smooth for $t \in [s, S)$ by the local unique existence of a strong solution\cite{CF}.   By this local existence of strong solutions,  we only need to show that
$\bold u(S) \in V$ for all such $s$.
Since the proof is exactly the same, we show it replacing $[s,S)$ with $[0,T)$ for convenience' sake.
This means $\bold u (t)$ is smooth for $t<T$.
By applying $\varphi = A\bold u$ to the equation \eqref{eq-ns-weak}, we get
\begin{equation} \label{eq-thm-11}
\frac{1}{2}\frac{d}{dt} \| A^{1/2} \bold u \|_{L^2}^2 + \nu \| A \bold u \|_{L^2}^2 = -b(\bold u, \bold u, A \bold u) + (\bold f, A \bold u).
\end{equation}
Since $(\bold f, A \bold u)$ does not make any technical difficulties, we assume $\bold f =0$.  For convenience, we denote
 $\bold w=\bold w_N$ and $\bold v =\bold v_N$.  By integration by parts, we have
\begin{equation}
\begin{aligned}
&\int (\bold w \, \nabla \bold u) \, A \bold u \, dx = \int (\bold w \, \nabla \bold u) \, P(-\Delta) \bold u \, dx \\
&\qquad= \int \bold w^i \, \partial_i \bold u^j \, (- \partial_k^2 \bold u^j ) \, dx = \int \partial_k\bold w^i \, \partial_i \bold u^j \, \partial_k \bold u^j \, dx.\\
\end{aligned}
\end{equation}

Hence, from the H\"{o}lder inequality and Theorem \ref{thm-tec-sob2}, we have
\begin{equation} \label{eq-thm-12}
\begin{aligned}
&\left| \int (\bold w \, \nabla \bold u) \, A \bold u \, dx \right| = \left| \int \partial_k\bold w^i \, \partial_i \bold u^j \, \partial_k \bold u^j \, dx \right| \le \| \nabla \bold w \|_{L^q} \| \nabla \bold u \|_{L^p}^2\\
&\qquad\le \| \nabla \bold w \|_{L^q} \| \nabla \bold u \|_{L^2}^{2a} \| \nabla^2 \bold u \|_{L^2}^{2-2a} \le C \| \nabla \bold w \|_{L^q}^{1/a} \| \nabla \bold u \|_{L^2}^{2} + \delta \| A \bold u \|_{L^2}^{2}
\end{aligned}
\end{equation}
Here $p$ and $a$ are chosen to satisfy $\displaystyle\frac{1}{q} + \frac{2}{p} = 1$ and $\displaystyle\frac{1}{p} = \frac{a}{2} + (1-a) \left( \frac{1}{2} - \frac{1}{3} \right)$. It follows that
\begin{equation}
\frac{3}{q} + \frac{2}{1/a} = \frac{3}{q} + 2(1-\frac{3}{2q})= 2
\end{equation}
are satisfied for $q >3$.

We also have by integration by parts,
\begin{equation}
\begin{aligned}
&\left| \int (\bold v \, \nabla \bold u) \, A \bold u \, dx \right| = \left| \int \partial_k\bold v^i \, \partial_i \bold u^j \, \partial_k \bold u^j \, dx \right|\\
&\quad\le \left| \int \partial_k\bold v^i \, \partial_i \bold v^j \, \partial_k \bold v^j \, dx \right| + \left| \int \partial_k\bold v^i \, \partial_i \bold v^j \, \partial_k \bold w^j \, dx \right| + \left| \int \partial_k\bold v^i \, \partial_i \bold w^j \, \partial_k \bold u^j \, dx \right|.
\end{aligned}
\end{equation}
Since
\begin{equation}
\begin{aligned}
&\left| \int \partial_k\bold v^i \, \partial_i \bold v^j \, \partial_k \bold w^j \, dx \right| + \left| \int \partial_k\bold v^i \, \partial_i \bold w^j \, \partial_k \bold u^j \, dx \right| \\
&\qquad\qquad\le \| \nabla \bold v \|_{L^p} \| \nabla \bold v \|_{L^p} \| \nabla \bold w \|_{L^q} + \| \nabla \bold v \|_{L^p} \| \nabla \bold w \|_{L^q} \| \nabla \bold u \|_{L^p} \\
&\qquad\qquad\le 2 \| \nabla \bold w \|_{L^q} \| \nabla \bold u \|_{L^p}^2,
\end{aligned}
\end{equation}
we get
\begin{equation} \label{eq-thm-13}
\left| \int \partial_k\bold v^i \, \partial_i \bold v^j \, \partial_k \bold w^j \, dx \right| + \left| \int \partial_k\bold v^i \, \partial_i \bold w^j \, \partial_k \bold u^j \, dx \right| \le C  \| \nabla \bold w \|_{L^q}^{1/a} \| \nabla \bold u \|_{L^2}^{2} + 2\delta \| A \bold u \|_{L^2}^{2}
\end{equation}
by repeating arguments in \eqref{eq-thm-12}.

Moreover, because of \eqref{eq-tec-v4}, we have
\begin{equation} \label{eq-thm-14}
\begin{aligned}
\left| \int \partial_k\bold v^i \, \partial_i \bold v^j \, \partial_k \bold v^j \, dx \right| &= \left| \int (\bold v \, \nabla \bold v) \, A \bold v \, dx \right| \le  \| \bold v \|_{L^4} \| \nabla \bold v \|_{L^4} \| A \bold v \|_{L^2}\\
&\le C N \| \bold v \|_{L^2}^{1/2} \| \nabla \bold v \|_{L^2} \| A \bold v \|_{L^2}^{3/2}\\
&\le C N^4 \| \bold v \|_{L^2}^2 \| \nabla \bold v \|_{L^2}^4 + \delta \| A \bold v \|_{L^2}^2.
\end{aligned}
\end{equation}

By combining \eqref{eq-thm-11}, \eqref{eq-thm-12}, \eqref{eq-thm-13}, and \eqref{eq-thm-14}, we have
\begin{equation}
\frac{1}{2}\frac{d}{dt} \| A^{1/2} \bold u \|_{L^2}^2 + \nu \| A \bold u \|_{L^2}^2 \le C N^4 \left( \| \nabla \bold w \|_{L^q}^{1/a} + \| \bold v \|_{L^2}^2 \| \nabla \bold v \|_{L^2}^2 \right) \| \nabla \bold u \|_{L^2}^{2} + 4 \delta \| A \bold u \|_{L^2}^{2} .
\end{equation}

Let $Y(t) = \| A^{1/2} \bold u \|_{L^2}^2$ and take $\delta = \nu/8$. Then we have
\begin{equation} \label{eq-thm-15}
\frac{d}{dt} Y(t) + \nu |A \bold u|^2 \le C N^4 \left( \| \nabla \bold w \|_{L^q}^{1/a} + \| \bold v \|_{L^2}^2 \| \nabla \bold v \|_{L^2}^2 \right) Y(t).
\end{equation}
From the above inequality \eqref{eq-thm-15} and the Gronwall's inquality, we have
\begin{equation}
Y(t) + \nu \int_0^T \| A \bold u(\tau) \|_{L^2}^2 d\tau \le Y(0) \exp\left\{ CN^4  \int_0^T\| \nabla \bold w \|_{L^q}^{1/a}
                  + \| \bold v \|_{L^2}^2 \| \nabla \bold v\|_{L^2}^2 d\tau \right\}
\end{equation}
and this inequality tells us that $\bold u$ is smooth.

\subsection{Proof for \eqref{eq-serrinN-1} with $q>3$}
For given $q>3$, choose $p$ and $\alpha$ to satisfy,
\begin{equation}
\frac{1}{q} + \frac{1}{p} = \frac{1}{2}, \, \text{ and } \, \frac{1}{p} = \frac{\alpha}{2} + (1-\alpha)\left(\frac{1}{2} - \frac{1}{3}\right).
\end{equation}
Then, from the H\"{o}lder inequality, we have
\begin{equation} \label{eq-thm-21}
\begin{aligned}
\int \bold w \nabla \bold u A \bold u dx &\le \| \bold w \|_{L^q} \| \nabla \bold u \|_{L^p} \| A \bold u \|_{L^2}\\
&\le \| \bold w \|_{L^q} \| \nabla \bold u \|_{L^2}^\alpha \| A \bold u \|_{L^2}^{2-\alpha}\\
&\le C \| \bold w \|_{L^q}^{2/\alpha} \| \nabla \bold u \|_{L^2}^2 + \delta \| A \bold u \|_{L^2}^2.\\
\end{aligned}
\end{equation}

We note that $q$ and $2/\alpha$ satisfy
\begin{equation}
\frac{3}{q} + \frac{2}{2/\alpha} = 1.
\end{equation}

By using \eqref{eq-tec-v4}, we have
\begin{equation} \label{eq-thm-22}
\begin{aligned}
\int \bold v \nabla \bold v A \bold u dx &\le C \| \bold v \|_{L^4} \| \nabla \bold v \|_{L^4} \| A \bold v \|_{L^2}\\
&\le  CN \| \bold v \|_{L^2}^{1/2} \| \nabla \bold v \|_{L^2} \| A \bold v \|_{L^2}^{3/2} \\
&\le  CN^4 \| \bold v \|_{L^2}^2 \| \nabla \bold v \|_{L^2}^4 + \delta \| A \bold v \|_{L^2}^2. \\
\end{aligned}
\end{equation}

By applying integration by parts twice to $\int \bold v \nabla \bold w A \bold w dx$, we have
\begin{equation} \label{eq-thm-23}
\begin{aligned}
\int \bold v \nabla \bold w A \bold w dx &= \sum_{i,j,k=1}^3 \int \bold v^i \partial_i \bold w^j (-\partial_k^2 \bold w^j) dx = \sum_{i,j,k=1}^3 \int \partial_k \bold v^i \partial_i \bold w^j \partial_k \bold w^j dx \\
&=- \sum_{i,j,k=1}^3 \left\{ \int \partial_k^2 \bold v^i \partial_i \bold w^j \bold w^j dx + \int \partial_k \bold v^i \partial_i \partial_k \bold w^j \bold w^j dx \right\}.
\end{aligned}
\end{equation}
Hence, by using the similar calculations in \eqref{eq-thm-21}, we can obtain
\begin{equation} \label{eq-thm-24}
\int \bold v \nabla \bold w A \bold w dx \le  C \| \bold w \|_{L^q}^{2/\alpha} \| \nabla \bold u \|_{L^2}^2 + \delta \| A \bold u \|_{L^2}^2.
\end{equation}
Finally, by integration by parts, we have
\begin{equation} \label{eq-thm-25}
\begin{aligned}
\int \bold v \nabla \bold w A \bold v dx &= \sum_{i,j,k=1}^3 \int \bold v^i \partial_i \bold w^j (-\partial_k^2 \bold v^j) dx \\
&= \sum_{i,j,k=1}^3 \left\{ \int \partial_k \bold v^i \partial_i \bold w^j \partial_k \bold v^j dx + \int \bold v^i \partial_i \partial_k \bold w^j \partial_k \bold v^j dx \right\}\\
&= \sum_{i,j,k=1}^3 \left\{ - \int \partial_k \bold v^i \bold w^j \partial_k \partial_i \bold v^j dx + \int \bold v^i \partial_i \partial_k \bold w^j \partial_k \bold v^j dx \right\}\\
\end{aligned}
\end{equation}

The last two terms are bounded by
\begin{equation} \label{eq-thm-26}
 C \| \bold w \|_{L^q}^{2/\alpha} \| \nabla \bold u \|_{L^2}^2 + C \| \bold v \|_{L^2}^2 \| \nabla \bold v \|_{L^2}^4 + 2 \delta \| A \bold v \|_{L^2}^2
\end{equation}
by the similar calculation in \eqref{eq-thm-21} and \eqref{eq-thm-22}.

By combining \eqref{eq-thm-11}, \eqref{eq-thm-21}, \eqref{eq-thm-22}, \eqref{eq-thm-24}, and \eqref{eq-thm-26}, we have
\begin{equation}
\frac{1}{2}\frac{d}{dt} \| A^{1/2} \bold u \|_{L^2}^2 + \nu \| A \bold u \|_{L^2}^2 \le CN^4 \left( \| \bold w \|_{L^q}^{2/\alpha} + \| \bold v \|_{L^2}^2 \| \nabla \bold v \|_{L^2}^2 \right) \| \nabla \bold u \|_{L^2}^{2} + 4 \delta \| A \bold u \|_{L^2}^{2} .
\end{equation}

Hence $\| A^{1/2} \bold u \|_{L^2}^2 $ is bounded and it implies that $\bold u$ is smooth.

\subsection{Proof for \eqref{eq-serrinN-1} with $q=3$}
We note that if $\bold w$ satisfies condition \eqref{eq-serrinN-1} with $q=3$, then the above calculation does not work because $a=0$.
We also remark that $\nabla {\bf w} \in L^{3/2,\infty}$ implies ${\bf w} \in L^{3,\infty}$.  So, we treat only the case \eqref{eq-serrinN-1} with $q=3$ separately.  We make use of \cite{SS} in this case.
We denote $\bold w_N= Q_N(\bold u)$ and $\bold v_N = P_N(\bold u_N) = (I-Q_N)(\bold u_N)$ as in \eqref{def-wN}.  This decomposition is in fact an orthogonal decomposition  and the equation satisfied by $\bold v$ is
\begin{equation} \label{eq-main-v}
\displaystyle\frac{d \bold v}{dt} + \nu A \bold  v = P_N \bold f - P_N B(\bold v + \bold w,\bold v + \bold w)
\end{equation}
and the equation satisfied by $\bold w$ is
\begin{equation} \label{eq-main-w}
\displaystyle\frac{d \bold w}{dt} + \nu A \bold  w = Q_N \bold f - Q_N B(\bold v + \bold w,\bold v + \bold w).
\end{equation}
We remark that the equation \eqref{eq-main-v} tell us that if $\bold w$ is regular, then the solution $\bold v$ is also regular. This procedure has not been exposed so far, but we will see how it operates more explicitly in the procedure of the proof.
The main idea for $q =3$ is to take $\varphi = A^{1/3} v$ in the equation \eqref{eq-ns-weak} instead of $A^{1/2}v$.   Then we get the following equation,
\begin{equation} \label{eq-prop-main-v1}
\frac{1}{2} \frac{d}{dt} \| A^{1/6} \bold v \|_{L^2}^2 + \nu \| A^{2/3} \bold v \|_{L^2}^2 = - b(\bold u,\bold u,A^{1/3} \bold v).
\end{equation}
\begin{prop} \label{prop-main-v1}
If $\bold w_N$ satisfies \eqref{eq-serrinN-1} with $q=3$, then
\begin{equation}
A^{1/6} \bold v \in L^\infty(0,T;H),\quad A^{2/3} \bold v \in L^2(0,T;H).
\end{equation}
\end{prop}

To prove the Proposition \ref{prop-main-v1}, we need to estimate the nonlinear term $b(\bold u, \bold u, A^{1/3} \bold v)$.
\begin{lemma}
\begin{equation}
\left| \int \bold v \, \nabla \bold u \, A^{1/3} \bold v \,dx \right| \le C \| \nabla \bold u \|_{L^2}^2 \| A^{1/6} \bold v \|_{L^2}^2 + \delta \| A^{2/3} \bold v \|_{L^2}^2.
\end{equation}
Here, $C$ is a  constant depending on $N$.
\end{lemma}
\begin{proof}
From \eqref{eq-tec-v1} and \eqref{eq-tec-v4}, we have
\begin{equation}
\| \bold v \|_{L^3} \le C_N \| A^{1/6} \bold v \|_{L^2}, \, \text{ and } \, \| A^{1/3} \bold v \|_{L^{6}} \le C_N \| A^{2/3} \bold v \|_{L^{2}}.
\end{equation}

From above inequalities and the Holder inequality, we have
\begin{equation}
\begin{aligned}
\left| \int \bold v \, \nabla \bold u \, A^{1/3} \bold v \,dx \right| &\le \| \bold v \|_{L^3} \| \nabla \bold u \|_{L^2} \| A^{1/3} \bold v \|_{L^{6}} \\
&\le C \| A^{1/6} \bold v \|_{L^2} \| \nabla \bold u \|_{L^2} \| A^{2/3} \bold v \|_{L^2} \\
&\le C \| \nabla \bold u \|_{L^2}^2 \| A^{1/6} \bold v \|_{L^2}^2 + \delta \| A^{2/3} \bold v \|_{L^2}^2.
\end{aligned}
\end{equation}
\end{proof}

\begin{lemma}
\begin{equation}
\left| \int \bold w \, \nabla \bold u \, A^{1/3} \bold v \,dx \right| \le C \| \nabla \bold u \|_{L^2}^2 \| \bold w \|_{L^{3}}^2 + \delta \| A^{2/3} \bold v \|_{L^2}^2.
\end{equation}
$C$ is a constant depending on $N$.
\end{lemma}
\begin{proof}
By the Holder inequality and lemma \ref{lem-tec-v1},
\begin{equation}
\begin{aligned}
\left| \int \bold w \, \nabla \bold u \, A^{1/3} \bold v \,dx \right| &\le \| \bold w \|_{L^3} \| \nabla \bold u \|_{L^2} \| A^{1/3} \bold v \|_{L^{6}} \\
&\le C \| \bold w \|_{L^{3}} \| \nabla \bold u \|_{L^2} \| A^{2/3} \bold v \|_{L^2} \\
&\le C \| \nabla \bold u \|_{L^2}^2 \|\bold w \|_{L^{3}}^2 + \delta \| A^{2/3} \bold v \|_{L^2}^2. \\
\end{aligned}
\end{equation}
\end{proof}

\begin{proof}[Proof of Proposition \ref{prop-main-v1}]
From \eqref{eq-prop-main-v1}, and from the above two lemmas, we have
\begin{equation}
\frac{d}{dt} \| A^{1/6} \bold v \|_{L^2}^2 + \nu \| A^{2/3} \bold v \|_{L^2}^2 \le C \| \nabla \bold u \|_{L^2}^2 \| A^{1/6} \bold v \|_{L^2}^2 + C \| \nabla \bold u \|_{L^2}^2 \| \bold w \|_{L^{3}}^2.
\end{equation}
by taking $\delta = \nu/4$.

Hence we get the following by Gronwall's inequality,
\begin{equation}
\begin{aligned}
&\| A^{1/6} \bold v \|_{L^2}^2 + \int_0^T \| A^{2/3} \bold v(\tau) \|_{L^2}^2 d\tau \\
&\qquad\le \exp\left\{ \int_0^T  C \| \nabla \bold u(\tau) \|_{L^2}^2 d\tau \right\} \left[ \| A^{1/6} \bold v(0) \|_{L^2}^2 \right. \\
 &\qquad \quad \left. + C  \sup_{0 \le \tau \le T} \| \bold w(\tau) \|_{L^{3}}^2 \int_0^T \| \nabla \bold u(\tau) \|_{L^2}^\gamma  d\tau \right].
\end{aligned}
\end{equation}
\end{proof}

Now, we are ready for the proof for $q=3$.  By Proposition \ref{prop-main-v1} and Lemma \ref{lem-tec-v1},
\begin{equation}
\sup_{0\le t\le T} \| \bold v (t) \|_{L^3}
\end{equation}
is bounded.  From this and our assumption, we have $\bold u \in L^\infty(0.T;L^3)$, which is enough by \eqref{eq-serrin-3}.

The following corollary actually implies that the energy spectrum should not be bounded uniformly by $\kappa^{-2-\epsilon}$.
\begin{cor}
Suppose $\bf u$ breaks down firstly at $t=T$.   Then, for any $N>1$ and any $\delta >2$,
\[ \liminf_{t\to T} \sup_{\kappa} \kappa^{\delta} E(\kappa) = \infty.\]
\end{cor}
\begin{proof}
Suppose not, then, for some $N>1$, $\delta>2$ and $t_0<T$, we get
\[ \sup_k \left( \kappa^{\delta} E(\kappa) \right) < C \]
for $t_0< t <T$ for some constant $C$.  But then, $E(\kappa) < C \kappa^{-\delta}$ for $t_0< t <T$.
Now, by theorem \ref{thm-main}, \ref{thm-tec-sob1},  and $L^2$ boundedness of $\bold u$, taking $q=3$, we have
\[  \infty =\int_0^T \| \bold w_N \|_{L^3} \leq C \int_0^T \| A^{1/4} \bold w_N \|_{L^2}.   \]
However,
\[  \| A^{1/4} \bold w_N \|_{L^2} = C \sum_{|k_i|>N} |k| |c^k |^2 \leq C \sum_{\kappa >N} \kappa E(\kappa) < C \]
and we arrive at a contradiction.
\end{proof}
Since the solution is smooth at any fixed $t<T$, the spectrum should be bounded by $\kappa^{-2-\epsilon}$ at any fixed $t<T$.
Thus, it actually implies that, as $t\to T$, there appears some range of $\kappa$ for which $E(\kappa) > C \kappa^{-2-\epsilon}$
for any $\epsilon$ and any $C>0$.

\end{document}